\documentclass{svmult}

\usepackage{helvet}         
\usepackage{courier}        
\usepackage{type1cm}        
\usepackage{makeidx}         
\usepackage{graphicx}        
\usepackage{multicol}        
\usepackage[bottom]{footmisc}

\makeindex             

\usepackage{amsmath,amssymb}
\usepackage{amsrefs}
\usepackage{hyperref}

\newcommand{\MRnumber}[1]{\href{http://www.ams.org/mathscinet-getitem?mr=#1}{\mbox{\bf~MR~#1}}}

\newcommand{\floor}[1]{\ensuremath{\lfloor #1 \rfloor}}

\newcommand{\ceiling}[1]{\ensuremath{\lceil #1 \rceil}}
\newcommand{\fracpart}[1]{\ensuremath{\{ #1 \}}}

\newcommand{\NN}{\ensuremath{{\mathbb N}}}

\newcommand{\RR}{\ensuremath{{\mathbb R}}}
\newcommand{\ZZ}{\ensuremath{{\mathbb Z}}}
\newcommand{\QQ}{\ensuremath{{\mathbb Q}}}

\newcommand{\al}{\alpha}

\newcommand{\len}{{\text{len}}}

\begin{document}
\title*{Can You Hear the Shape of a Beatty Sequence?}
\author{Ron Graham, Kevin O'Bryant}
\institute{Ron Graham at University of California, San Diego, \href{mailto:graham@ucsd.edu}{graham@ucsd.edu}
    \and Kevin O'Bryant at The City University of New York, College of Staten Island and Graduate Center,
        \href{mailto:kevin@member.ams.org}{kevin@member.ams.org}.
        This work was supported (in part) by a grant from The City University of New York PSC-CUNY Research Award Program.}
\maketitle

\abstract{
Let $K(x_1,\dots,x_d)$ be a polynomial. If you are not given the real numbers
$\al_1,\al_2,\dots,\al_d$, but are given the polynomial $K$ and the sequence
$a_n=K(\floor{n\al_1},\floor{n\al_2},\dots,\floor{n\al_d})$, can you deduce the values of
$\al_i$? Not, it turns out, in general. But with additional irrationality hypotheses and certain
polynomials, it is possible. We also consider the problem of deducing $\al_i$ from the integer sequence $(\floor{\floor{\cdots \floor{\floor{n\al_1}\al_2}\cdots \al_{d-1}}\al_d})_{n=1}^\infty$.
}

\section{Introduction}
If you are given a sequence of integers $(a_n)_{n=1}^\infty$ and told that the sequence was
generated by the formula $a_n=\floor{n\al_1}\floor{n\al_2}$ for some real numbers
$\al_1,\al_2$, is it possible to determine $\al_1$ and $\al_2$? In other words, what
are the solutions $(\al_1,\al_2,{\beta}_1,{\beta}_2)$ to the infinite system of equations
    \[ \floor{n\al_1}\floor{n\al_2} =\floor{n{\beta}_1}\floor{n{\beta}_2}\qquad (n\in \NN)?\]

A {\em generalized polynomial} is defined to be any formula built up from the unknowns $x_1,x_2,\dots$, the real numbers, and the operations of addition, multiplication, and the floor function. These have arisen recently in ergodic theory (e.g., \cites{MR2318563,MR2246589,MR1292518}), particularly in connection with rotations on nilmanifolds.

The first problem we are concerned with is, given a sequence $(a_n)_{n=1}^\infty$ of integers and a generalized polynomial $G(\bar{x})$, to describe the set of $\bar{\al}\in\RR^d$ such that
    \[\forall n\geq 1, \qquad G({n\bar\al})=a_n.\]
A few examples will help to clarify the difficulty in dealing with generalized polynomials. First, we note that to determine real numbers from an integer sequence, we must use the tail of the sequence, i.e., limits must be involved in some form. As a first example, consider the sequence $a_n=n-1$  and the generalized polynomial $G(\bar x)=\floor{x_1}+\floor{x_2}$. For {\em any} irrational $\al_1$ and $\al_2=1-\al_1$, we have \(G(n\al_1,n\al_2)=a_n\)for all positive integers $n$. Another curious example is  given by $G(x_1,x_2,n) = \floor{\floor{n\,x_1}\,x_2}$, which satisfies (among very many other sporadic relations)
    \[\forall n \in\ZZ, \qquad G(3/7,2/9,n) = G(1/3,2/7,n).\]
I.\ H{\aa}land Knutson [personal communication] notes that
    \[G(n)=\floor{\floor{\sqrt2 n}2\sqrt2 n}-\floor{\sqrt2 n}^2 -2n^2+1 = \left\{
                                                                            \begin{array}{ll}
                                                                              1, & \hbox{$n=0$;} \\
                                                                              0, & \hbox{$n\in \ZZ \setminus\{0\}$.}
                                                                            \end{array}
                                                                          \right.
    \]
In this work, we restrict ourselves to generalized polynomials with a particular structure.

Specifically, let $K(\bar x)$ be a (classical) polynomial, and set $a_n=K(\floor{n \bar \beta})$ (the floor function applied to each component of the vector $\bar \beta$) for some `sufficiently' irrational $\bar \beta$. We attempt to find all nontrivial solutions to the system of equations
    \[\forall n \geq 1, \qquad K(\floor{n \al}) = a_n.\]
With varying success we treat linear polynomials $x_1+\cdots+x_d$, sums of powers $x_1^r+\cdots+x_d^r$, and monomials $x_1\cdots x_d$, and other shapes.

The second problem we address is, given $d$ and a sequence $(a_n)_{n=1}^\infty$ of integers, to find all solutions to the infinite system of equations
    \[\floor{\floor{\cdots \floor{\floor{n\al_1}\al_2} \cdots \al_{d-1}}\al_d} = a_n.\]

We were motivated by two problems\footnote{It is plausible that their origins were in signal analysis. Consider a linear signal $(\al t + \gamma)_{t\in\RR}$, that is measured at discrete times (replace $t\in\RR$ with $n\in \ZZ_{\ge0}$) and with finite precision (replace $\al n +\gamma$ with $\floor{\al n +\gamma}$). Given finitely many such measurements, how accurately can you estimate $\al$? It is not difficult to imagine a situation where several such signals are preprocessed algebraically into a single signal, and yet one still wishes to discern the original signals.} given in ``Concrete Mathematics''~\cite{CM}:
    \begin{quote}
    {\bf Comment to Bonus Problem 3.49:} Find a necessary and sufficient condition on the real numbers
    $0\le \al<1$ and $0\le {\beta}<1$
    such that we can determine the unordered pair $\{\al,{\beta}\}$ from the infinite multiset of values
    \(\big\{\floor{n\al}+\floor{n{\beta}} \mid n>0\big\}\).

    {\bf Research Problem 3.50:} Find a necessary and sufficient condition on the nonnegative real
    numbers $\al$ and ${\beta}$ such that we can determine $\al$ and ${\beta}$ from the infinite
    multiset of values \(\big\{ \floor{\floor{n\al}{\beta}} \big\}\).
    \end{quote}
A partial solution to the first problem (with the additional assumption that $1,\al,\beta$ are linearly independent over $\QQ$) has recently been published~\cite{Rasmussen}, and \cite{CM} itself credits a sufficient condition for the second problem to unpublished notes of William A.\ Veech. We provide partial answers to generalizations of both problems.

To state our theorems, it is convenient to first introduce some notation. For a vector of reals $\bar{\zeta}=\langle \zeta_1,\dots,\zeta_d \rangle$, we define the fractional part $\fracpart{\bar{\zeta}}=\langle \fracpart{\zeta_1},\dots,\fracpart{\zeta_d}\rangle$ (this paper contains no sets of vectors!) and floor $\floor{\bar{\zeta}}=\langle \floor{\zeta_1},\dots,\floor{\zeta_d}\rangle$. Also, inequalities such as $\bar{\zeta}\ge0$ are to be understood componentwise, i.e., $\zeta_1\ge0, \dots,\zeta_d\ge 0$. We say that $\bar{\zeta}$ is rational if there is a nonzero vector of integers $\bar{c}$ such that the dot product $\bar{c}\cdot \bar{\zeta}$ is an integer, and otherwise say that $\bar{\zeta}$ is irrational. For a polynomial $K(x_1,\dots,x_d)$, the expression $K(\bar{\zeta})$ is defined to be $K(\zeta_1,\dots,\zeta_d)$. Also, $\sum \bar{\zeta}=\zeta_1+\cdots+\zeta_d$.

Let $\bar\zeta,\bar\eta\in\ZZ^d$ both sum to 0, and let $\sigma$ be a permutation of $1,2\dots,d$. Let ${\beta}_i=\al_{\sigma(i)}+\zeta_i$ and $\delta_i=\gamma_{\sigma(i)}+\eta_i$. Then trivially
    \[\floor{n\al_1+\gamma_1}+\dots+\floor{n\al_d+\gamma_d} =\floor{n{\beta}_1+\delta_1}+\dots+\floor{n{\beta}_d+\delta_d}  \]
for all $n$. Our first theorem states that this is the only type of solution that is possible when $\bar{\al}$ is irrational. It is plausible and consistent with our experiments that the phrase ``$\bar\al$ is irrational'' could be weakened to ``$\alpha_i+\alpha_j$ is not an integer for any $i,j$''.

\begin{theorem}\label{S11}
Let $K(x_1,\dots,x_d)=x_1+\dots+x_d$, and $\bar\al,\bar\gamma,\bar {\beta},\bar\delta\in\RR^d$. If
    \[\forall n\ge 1, \qquad K(\floor{n\bar{\al}+\bar{\gamma}})=K(\floor{n\bar{{\beta}}+\bar{\delta}}),\]
then either $\bar \al$ is rational, or there are lattice points $\bar \zeta,\bar\eta\in\ZZ^d$ and a permutation $\sigma$ of $1,2\dots,d$ with ${\beta}_i=\al_{\sigma(i)}+\zeta_i$, $\delta_i=\gamma_{\sigma(i)}+\eta_i$, and $\sum \bar \zeta=\sum\bar\eta=0$.
\end{theorem}

Using the fact that for non-integral $\alpha$, the sequence $(|\floor{n\alpha}|)_{n=1}^\infty$ contains arbitrarily large primes, we can also handle products. Note that in this case we do not need the irrationality of $\bar\al$.
\begin{theorem}\label{ProductForm}
Let $K(\bar x)= x_1x_2\cdots x_d$, and $\bar\al,\bar {\beta}\in\RR^d$. If
    \[\forall n\ge 1, \qquad K(\floor{n\bar{\al}})=K(\floor{n\bar{{\beta}}}),\]
then either some $\alpha_i$ is an integer or $\{\alpha_1,\dots,\alpha_d\}=\{\beta_1,\dots,\beta_d\}$ (as multi-sets).
\end{theorem}

The next theorem assumes algebraic independence of the $\al_i$, but this is used in only a very weak manner. The hypothesis could be weakened to assuming that $\bar\al$ is irrational and the $\al_i$ do not satisfy any of a specific (depending on $K$) small finite set of algebraic relations. In fact, we believe that the conclusion is true as long as none of $\al_i$ are integers. Additionally, whether a particular form for $S$ can be included in the following theorem depends on an {\em ad hoc} solution of a system of equations that arises. Certainly the given list is not the extent of the method, but a general statement remains elusive.
\begin{theorem}\label{Skk}
Let $K(\bar x)=S(\bar x)+R(\bar x)$ be a polynomial, where $S(x)$ is a symmetric polynomial of one the following types ($d\ge 2, r\ge 2$)
    \[\prod_{i=1}^d x_i, \quad \sum_{i=1}^d x_i^r ,\quad\text{or} \quad\sum_{i,j=1}^d x_ix_j,\]
and $\deg(R)<\deg(S)$. Assume that $\al_i$ ($1\le i \le d$) are positive and do not satisfy any algebraic relations of degree less than $\deg(S)$, and ${\beta}_i$  ($1\le i \le d$) are positive and do not satisfy any algebraic relations of degree less than $\deg(S)$. If
    \[\forall n\ge 1, \qquad K(\floor{n\bar\al})=K(\floor{n\bar {\beta}}),\]
then $\{\alpha_i \colon 1\le i \le d\}=\{{\beta}_i \colon 1\le i \le d\}$.
\end{theorem}

Rasmussen~\cite{Rasmussen} proves the $d=2$ and $d=3$ cases of the following conjecture:
\begin{conjecture}\label{conj:nested}
Suppose that $\bar{\al},\bar{\beta}\in[0,1)^d$, and that both $\langle \al_1, \al_1\al_2,\ldots,\al_1\al_2\cdots \al_d\rangle$ and $\langle \beta_1, \beta_1\beta_2,\ldots,\beta_1\beta_2\cdots \beta_d\rangle$ are irrational. If
    $$\floor{ \cdots \floor{\floor{n\al_1}\,\al_2}\,\cdots \al_d} = \floor{ \cdots \floor{\floor{n \beta_1}\,\beta_2}\,\cdots \beta_d} $$
for all $n\ge 1$, then $\bar{\al}=\bar{\beta}$.
\end{conjecture}
We give his proofs (with corrections) in Section~\ref{sec:Rasmussen}. It is certainly desirable to extend his work to $d> 3$, to weaken the irrationality condition, and to consider $\al_i\in\RR$ instead of merely $\al_i\in[0,1)$. Using a different method, we make the following step in this direction.

\begin{theorem}\label{Snested2}
Suppose that $\bar{\al},\bar{\beta} \in [1,\infty)\times[2,\infty)^{d-1}$ are irrational. If
    \[\forall n \geq 1, \qquad \floor{ \cdots \floor{\floor{n\al_1}\,\al_2}\,\cdots \al_d} = \floor{ \cdots \floor{\floor{n\beta_1}\,\beta_2}\,\cdots \beta_d},\]
then the sets of fractional parts are equal: $\{ \{\al_1\},\dots,\{\al_d\}\} = \{\{\beta_1\},\dots,\{\beta_d\}\}$.
\end{theorem}

\section{Proofs}

\subsection{Proof of Theorem~\ref{S11}}

\begin{proof}
Without loss of generality, we assume that $\bar\al,\bar \gamma$ are in $[0,1)^d$ and that $\bar\al$ is irrational. Let $S(i)=K(\floor{n\bar\al+\bar\gamma})$, and set. Define $\Delta(i)=S(i+1)-S(i)$. Thus $\Delta(i)\in \{0,1,\dots,k\}$. We say that $S$ has an $r$-jump at $i$ if $S(i+1)-S(i)=\Delta(i)=r$. The frequency
of $r$-jumps of $S$ depends on the frequency that $(\fracpart{n\al_1+\gamma_1},\dots,\fracpart{n\al_k+{\beta}_k})$ is in a particular subcube of $[0,1)^k$. To wit, if there are exactly $r$ coordinates $j$ such that
    $$1-\al_j \le \fracpart{i\al_j+\gamma_j} <1,$$
which is equivalent (ignoring the technical circumstance when $1-\al_j-\gamma_j<0$) to
    $$1-\al_j-\gamma_j \le \fracpart{i\al_j} <1-\gamma_j,$$
then there is an $r$-jump at $i$. The volume of this region in $[0,1)^d$ is the asymptotic frequency of $r$-jumps of $S$, and is given by
    $$V_r = \sum_{\substack{R\subseteq K \\ |R|=r}} \, \prod_{i\in R} (1-\al_i)\prod_{j\in K\setminus R}  \al_j\quad \text{where $K=\{1,2,\dots,k\}$}.$$
Consider the polynomial
    $$P(z)=\prod_{i=1}^k \left\{(1-\al_i)z+\al_i\right\}=\sum_{r=0}^k V_r z^r,$$
which is determined by $S$. Hence, all the roots $-\frac{\al_i}{1-\al_i}$ of $P$ are determined by $S$, and
therefore, so are all the values $\al_i$.

Let $i_0,i_1,\dots$ be the sequence of $i$ such that $\Delta(i)=k$, which is exactly the same condition as `for all $j$, $1-\al_j-\gamma_j \le \fracpart{i\al_j} < 1-\gamma_j$'. By the irrationality of $\bar\al$, the closure of
    $$\{(\fracpart{i_t\al_1},\dots,\fracpart{i_t\al_k}):t=0,1,2,\dots\}$$
is the set
    $$\prod_{t=1}^k [1-\al_j-\gamma_j,1-\gamma_j].$$
Since we already know the $\al_j$, we find that the $\gamma_j$ are also determined.
\end{proof}

\subsection{Proof of Theorem~\ref{ProductForm}}

\begin{lemma}\label{lem:Beatty}
If $\alpha\in\RR$ is not an integer, then the sequence $(|\floor{n\alpha}|)_{n=1}^\infty$ of nonnegative integers contains arbitrarily large prime numbers.
\end{lemma}

Our proof works equally well to show that $(|\floor{n\alpha+\gamma}|)_{n=1}^\infty$ contains large primes when $\alpha$ is irrational, but for rational $\alpha$ the conclusion would be false: the sequence \mbox{$(\floor{n\frac{15}{2}+3})_{n=1}^\infty$} contains only one prime.

\begin{proof}
First, observe that the sequence contains {\em all} large positive integers if $0<|\alpha|\leq 1$, so we assume henceforth that $|\alpha|>1$.

First, we further assume that $\alpha$ is irrational and positive. We will show that $(\floor{n\alpha+\gamma})_{n=1}^\infty$ contains arbitrarily large primes. We note the oft-used and elementary criterion~\cite{FraenkelsPartition} that $k\in (\floor{n\alpha+\gamma})_{n=1}^\infty$ if and only if $k\geq\floor{\alpha+\gamma}$ and either $\fracpart{(k-\gamma)/\al}>1-1/\al$ or $(k-\gamma)/\al\in\ZZ$. Thus it suffices for our purposes to show that the sequence of fractional parts $\{p/\alpha\}$ is uniformly distributed, where $p$ goes through the prime numbers. This was shown by Vinogradov~\cite{Vinogradov}*{Chapter XI}.

If $\al$ is irrational and negative, then $|\floor{n\alpha}|=\floor{n|\alpha|+1}$, and this is the case considered in the previous paragraph.

For the remainder of the proof, we assume that $\alpha=q/p$, with $p\ge 2$ and $\gcd(p,q)=1$. In particular,
    \[\floor{n\alpha} = \floor{\frac{ nq }{p}}.\]
It suffices for our purpose to restrict to $n\equiv r \pmod p$, that is, we replace $n$ with $np+r$:
    \[\floor{\frac{(np+r)q}{p}} = nq + \floor{\frac{rq}p}.\]
We have reduced the problem (by Dirichlet's theorem on the infinitude of primes in arithmetic progressions) to choosing $r$ so that
    \(\gcd\left( q ,\floor{rq/p}\right) = 1.\)
Set $r=q^{-1}$, where $q^{-1}$ is the integer in $[2,p+1]$ with $qq^{-1}\equiv 1 \pmod p$; define $u$ through $q q^{-1} = pu +1$, and note that $\gcd(q,u)=1$. We now have
    \[\floor{\frac{rq}p}= \floor{ \frac{ q^{-1}q}p } =\floor{u+\frac 1p}= u,\]
with the last equality being our usage of $p\ge 2$, i.e., the reason we need $\alpha$ to be nonintegral. Since $\gcd(q,u)=1$, we have
    \(\displaystyle \gcd\left( q ,\floor{rq/p}\right) = \gcd\left( q, u \right) = 1.\)
\end{proof}

\begin{proof}[Proof of Theorem~\ref{ProductForm}] We proceed by induction on $d$. The claim is immediate for $d = 1$. Now assume that $d\geq 2$ and that Theorem~\ref{ProductForm} holdes for $d-1$.

Assume without loss of generality that $\al_1\geq \al_2 \geq \cdots \al_d$.
If $\floor{n\al_1} = q$ is prime, then it will show up in the factorization of
$\prod_{i=1}^{d} \floor{n\al_i} = P_n$ as a prime factor $q \geq P_n^{1/d}$
(since $\floor{n\al_1}\geq \floor{n\al_i}$ for all $i$).
Conversely, any prime factor $q$ of $P_n$ which is greater than or equal to $P_n ^{1/d}$ must come from
$\floor{n\al_1}$. Thus, we know the value of $\floor{n\al_1}$ for
infinitely many values of $n$, and so we can determine $\al_1$. Now, by factoring out $\floor{n\al_1}$
from each term $K(\floor{n\bar\al})$, we have reduced the problem to the case of $d-1$ factors. This completes the induction step, and the theorem is proved.
\end{proof}

\subsection{Proof of Theorem~\ref{Skk}}

A $d$-dimensional cube is defined as $Q_a(\bar x) :=\{a+\sum_{j=1}^d \epsilon_j x_j \colon \epsilon_j\in\{0,1\}\}$.
\begin{lemma}
Let $d\in \NN$, and $a,b\in\RR, \bar x, \bar y\in \RR^d$. If $Q=Q_a(\bar x)=Q_b(\bar y)$ and $|Q|=2^d$, then $\{|x_j| \colon 1\le j \le d\} = \{|y_j|\colon 1\le j \le d\}$.
\end{lemma}

\begin{proof}
Since $|Q|=2^d$, we know that none of $x_j,y_j$ are 0, and that the $x_j$ are distinct, as are the $y_j$. Further, note that,
    \[Q=Q_a(x_1,\dots,x_d) = Q_{\min Q}(|x_1|,\dots, |x_d|),\]
so that we can assume without loss of generality that $x_j,y_j$ are positive, and that $a=b=\min Q$.

The generating function of $Q$ factors as
    \[f(z)= \sum_{q\in Q} z^q = z^a \prod_{j=1}^d (1+z^{x_j}) = z^a \prod_{j=1}^d (1+z^{y_j}).\]
whence
    \begin{equation}\label{equ:productequal}
         \prod_{j=1}^d (1+z^{x_j}) = \prod_{j=1}^d (1+z^{y_j})
    \end{equation}
for appropriate complex numbers $z$.

We will show by induction on $d$ that such an equality implies that $\{x_j \colon 1\leq j \leq d\}=\{y_j \colon 1\leq j \leq d\}$. This is trivially true for $d=1$. Now assume that it is true for $d-1\geq 1$.

Let $X=\max\{x_1,\dots,x_d\}, Y=\max\{y_1,\dots,y_d\}$. The left hand side of Equ.~\eqref{equ:productequal} vanishes at $z=\exp(\pi i/X)$, and so the right hand side must also vanish, i.e., $1+\exp(\pi i y_j /X)=0$ for some $j$. It follows that $y_j/X = 2k+1$ for some integer $k$, and therefore that for some $j$, $Y\geq y_j \geq X$. Interchanging the roles of $x$ and $y$ yields that some for some $j$, $X\geq x_j \geq Y$, and therefore $X=Y$. We can cancel out the terms on the left and right hand sides of Equ.~\eqref{equ:productequal} corresponding to $X$ and $Y$ (which are the same), and we get a product with $d-1$ factors, completing the inductive step.
\end{proof}

\begin{proof}[Proof of Theorem~\ref{Skk}]
Define
    \[\Delta(n)=\frac{K(\floor{(n+1)\bar\al})-K(\floor{n\bar\al})}{n^{D-1}}.\]
The set $\{\Delta(n)\colon n\in\NN\}$ has limit points (call the set of limit points $\Delta$) which only depend on $S$ and which we can describe in the following manner:
    \[\Delta = \left\{\sum_{i=1}^d [\al_i]\,\frac{\partial S}{\partial x_i} (\bar\al) \colon [\al_i]\in\{\floor{\al_i},\ceiling{\al_i}\}\right\}.\]
We have assumed that $\bar\al$ is irrational to guarantee that all of these expressions arise as limit points, and we assumed that $\al_i$ are algebraically independent to guarantee that all of these expressions correspond to distinct real numbers. We can apply the previous lemma to learn
    \[L_S:=\left\{\left|\frac{\partial S}{\partial x_i} (\bar\al)\right|\right\}.\]

From here, we apply {\em ad hoc} arguments that depend on the special structure of $S$.

If $S(\bar x) = \prod_{i=1}^d x_i$, then we have learned
    \[ L=\{ \al_j^{-1} \prod_{i=1}^d \al_i \colon 1\le j \le d\} .\]
The product of all the elements of this set is just
    \[\left(\prod_{i=1}^d \al_i\right)^{d-1}.\]
As $\bar\al >0$, we can take the $(d-1)$-th root, learning the value of $\prod \al_i$. Dividing $\prod \al_i$ by each element of the set $L$ yields the set
    \[\{\al_j \colon 1\le j \le d\}.\]

If $S(\bar x) = \sum_{i=1}^d x_i^r$, then we have learned
    \[ L=\{r \al_j^{r-1} \colon 1\le j \le d\} \]
Dividing each element of $L$ by $r$ and then taking $(r-1)$-th roots (again using $\bar\al>0$) yields the set
    \[\{\al_j \colon 1\le j \le d\}.\]

If $K(\bar x) = \sum_{i,j=1}^d x_ix_j$, then we have learned
    \[ L=\{ \alpha_i+\sum_{j=1}^d \alpha_j \colon 1\le i \le d\}. \]
The sum of all the elements of this set is just
    \[(d+1) \sum_{j=1}^d \alpha_j.\]
Dividing by $d+1$ yields $\sum \alpha_j$, and subtracting this from each element of $L$ gives the set
    \[\{\al_i \colon 1\le i \le d\}.\]
\end{proof}

\subsection{Rasmussen's Approach to Conjecture~\ref{conj:nested}}\label{sec:Rasmussen}

Our first proof of the $d=2$ case is markedly different from the other proofs of this article. First, we do not assume $\langle \al_1,\al_2 \rangle$ to be
irrational, but $\langle \al_1,\al_1\al_2 \rangle$. Second, the proof is by contradiction and therefore not constructive.

Suppose, by way of contradiction, that
    $$
    s(n)=\floor{\floor{n\al_1}\al_2}=\floor{\floor{n{\beta}_1}{\beta}_2},
    $$
with $\bar\al \not=\bar \beta$, and $\langle \al_1,\al_1\al_2\rangle, \langle {\beta}_1,{\beta}_1{\beta}_2\rangle$ are irrational. Note
    $$
    \al_1\al_2 = \lim_{n\to \infty} \frac{s(n)}{n}={\beta}_1{\beta}_2.
    $$
Suppose without loss of generality that ${\beta}_2<\al_2$ and $\al_1<{\beta}_1$. Since $\langle \al_1,\al_1\al_2 \rangle$ is irrational, there exists an $n$ such that $\fracpart{n\al_1}>\frac{\al_2+{\beta}_2}{2\al_2}$ (note that $\frac{\al_2+{\beta}_2}{2\al_2}<1$ by virtue of the assumption that ${\beta}_2<\al_2$) and ${\beta}_2<\fracpart{n\al_1\al_2}<\frac{\al_2+{\beta}_2}{2}$. But then
    $$s(n)=\floor{\floor{n\al_1}\al_2}=\floor{n\al_1\al_2-\fracpart{n\al_1}\al_2}=\floor{n\al_1\al_2}-1$$
whereas, since $\fracpart{n{\beta}_1{\beta}_2}=\fracpart{n\al_1\al_2}>{\beta}_2>\fracpart{n{\beta}_1}{\beta}_2$,
    $$s(n)=\floor{\floor{n{\beta}_1}{\beta}_2}=\floor{n{\beta}_1{\beta}_2-\fracpart{n{\beta}_1}{\beta}_2}=\floor{n{\beta}_1{\beta}_2}=\floor{n\al_1\al_2}.$$

The method of Rasmussen, which works\footnote{In the $d=3$ case, Rasmussen miswrote the formula for $T_{3,2}$, which erroneously led to a system of equations (using $T_{3,1}$ and $T_{3,2}$) with a unique solution. The analogous system using $T_{3,1}$ and $T_{3,3}$, however, does have a unique solution. We give this minor correction here.} for $d=2$ and $d=3$, might be more amenable to generalization. Define for $\bar{\al} \in \RR^d$
    \[T_{d,k}:= \lim_{N\to\infty} \frac 1N \sum_{n=1}^N \left(n \al_1\cdots \al_d -  \floor{ \cdots \floor{\floor{n\al_1}\,\al_2}\,\cdots \al_d} \right)^k.\]
Using Weyl's Criterion and straightforward integration (with which we trust {\em Mathematica 6.0}), we find that if $\bar\al \in [0,1)^d$ and $\langle \al_1, \al_1\al_2,\ldots, \al_1\al_2\cdots \al_d\rangle\in [0,1)^d$ is irrational, then
    \begin{align*}
    T_{2,1} &= \frac{1+\al_2}{2}, \\
    T_{3,1} &= \frac{1+\al_3+\al_2\al_3}{2}, \\
    T_{3,3} &= \frac 12 T_{3,1} \cdot \left((1+\al_3+\al_3^2)+(\al_3+\al_3^2)\al_2 + (\al_3^2)\al_2^2\right).
    \end{align*}
Since both $P_d:=\prod_{i=1}^d \al_i$ and the $T_{d,k}$ are determined by the sequence
    \[(\floor{ \cdots \floor{\floor{n\al_1}\,\al_2}\,\cdots \al_d})_{n=1}^\infty,\]
so are the $\al_i$: for $d=2$
    \[\al_2 = 2T_{2,1}-1, \qquad \al_1 = P_d/\al_2\]
and for $d=3$
    \begin{align*}
    s &= \text{sgn}\left(4 T_{3,1}^3-2 T_{3,1}^2+T_{3,1}-2 T_{3,3}\right), \\
    \al_2 &= \frac{-4 T_{3,1}^3-T_{3,1}+4 T_{3,3}+s \left(1-2
   T_{3,1}\right) \sqrt{-12 T_{3,1}^4+4 T_{3,1}^3-3
   T_{3,1}^2+8 T_{3,3} T_{3,1}}}{2 \left(4 T_{3,1}^3-2
   T_{3,1}^2+T_{3,1}-2 T_{3,3}\right)}, \\
    \al_3 &= \frac{2 T_{3,1}-1}{1+\al_2},\\
    \al_1 &= \frac{P_3}{\al_2\al_3}.
    \end{align*}

We expect that this approach will work in principle for arbitrarily large $d$, but the practical difficulties in carrying this out are not trivial. Already, we are loathe to check the formula for $T_{3,3}$ and to solve the resulting equations by hand. {\em Mathematica}'s {\tt Solve} command only gives generic solutions, while its {\tt Reduce} command is too slow to handle $d=4$.

The formulas given above for $T_{d,k}$ can be computed using Weyl's criterion: If $\bar\al$ is irrational, then
    \[\frac 1N \sum_{n=1}^N f(\fracpart{n\bar \al}) = \int_{[0,1)^d} f(\bar x) d\bar x.\]
We calculate $T_{3,1}$ as an example. By repeatedly using $\floor{q}=q-\fracpart{q}$ and $\fracpart{q+r}=\fracpart{\fracpart{q}+r}$, we calculate
    \begin{align*}
     \left(n \al_1\al_2 \al_3 \right.  &  \left. -  \floor{\floor{\floor{n\al_1}\,\al_2} \al_3} \right)\\
    &= \fracpart{n\al_1}\al_2\al_3 + \fracpart{\fracpart{n\al_1\al_2}-\fracpart{n\al_1}\al_2}\al_3 \\
    &\qquad\qquad + \fracpart{\fracpart{n\al_1\al_2\al_3} -  \fracpart{n\al_1}\al_2\al_3-\fracpart{\fracpart{n\al_1\al_2} - \fracpart{n\al_1}\al_2}\al_3}\\
    &= x\al_2\al_3 + \fracpart{y - x \al_2}\al_3  + \fracpart{z - x \al_2\al_3-\fracpart{y - x\al_2}\al_3} \\
    \end{align*}
where $\langle x,y,z\rangle = \langle \fracpart{n\al_1},\fracpart{n\al_1\al_2},\fracpart{n\al_1\al_2\al_3}\rangle$. By Weyl's criterion, we get
    \begin{multline*}
    T_{3,2}=\lim_{N\to\infty} \frac1N\sum_{n=1}^N \left(n \al_1\al_2 \al_3 -  \floor{\floor{\floor{n\al_1}\,\al_2} \al_3} \right) \\
        \int_0^1 \int_0^1\int_0^1 x\al_2\al_3 + \fracpart{y - x \al_2}\al_3  + \fracpart{z - x \al_2\al_3-\fracpart{y - x\al_2}\al_3} \,dx\,dy\,dz.
    \end{multline*}
Using $\bar\al \in [0,1)^3$, we can eliminate the fractional parts in the above integral and get
    \[T_{3,2}=\frac13+\frac{1+\al_2}{2}\al_3 + \frac{2+3\al_2+2\al_2^2}{6} \al_3^2.\]
It is clear that this method can yield a formula for $T_{d,k}$ for any $d,k$.

\subsection{Proof of Theorem~\ref{Snested2}}

Let $[x]_0$ be the floor of $x$, and $[x]_1$ be the ceiling. Let
    $$T(W,\bar{\al};n):=[ \dots [[n\al_1]_{w_1}]\al_2]_{w_2}\dots\al_{d}]_{w_{d}},$$
where $W=w_1w_2\dots w_k$ is a word in the alphabet $\{0,1\}$, and $\bar{\al}=\langle \al_1,\al_2,\dots,\al_d\rangle$. In addition to its usual meaning, let ``$<$'' denote the lexicographic ordering on $\{0,1\}^d$. Let $h(W)$ be the Hamming weight of the word $W$, i.e., the number of {\tt 1}s in $W$.

\begin{lemma}\label{lem:Bigalpha}
If $\al_1,\dots,\al_d$ are not integers, with $\al_1>1$ and $\al_i>2$ (for $2\le i \le d$), then
    $$W<V \Leftrightarrow T(W,\bar{\al};1) < T(V,\bar{\al};1).$$
\end{lemma}

\begin{proof}
We work by induction on $d$. For $d=1$, the result obviously holds since $\al_1\not\in\ZZ$.

Now assume that $d\geq2$ and that the result holds for {\em all} $(d-1)$-tuples . Assume that
$W<V$. If $w_1=v_1$, then we may apply the induction hypothesis by observing that
    \begin{align*}
    T(W,\langle \al_1,\dots,\al_d \rangle;1)
        &= T(w_2\cdots w_d, \langle [\al_1]_{w_1}\al_2,\al_3,\dots,\al_d \rangle;1)\\
    T(V,\langle \al_1,\dots,\al_d \rangle;1)
        &= T(v_2\cdots v_d, \langle [\al_1]_{w_1}\al_2,\al_3,\dots,\al_d \rangle;1)
    \end{align*}

Thus, we may assume that $w_1<v_1$, and so  $w_1\cdots w_{d-1} < v_1 \cdots v_{d-1}$. Since $\al_d>2$, we
have $[m\al_d]_{w_d}<[m^\prime \al_d]_{v_d}$ whenever $0<m<m^\prime$, and by induction we have
    $$
    m=T(w_1 \cdots w_{d-1},\langle \al_1,\dots,\al_{d-1}\rangle; 1)<T(v_1\cdots v_{d-1},\langle \al_1,\dots,\al_{d-1}\rangle;1)= m^\prime.
    $$
Now, we have
    $
        T(W,\bar{\al}; 1)
            =  [m \al_d]_{w_d}
            <  [m^\prime \al_d]_{v_d}
            =  T(V,\bar{\al}; 1).
    $
\end{proof}

\begin{lemma} \label{lem:BigDeal}
Suppose that $\bar\al,\bar\beta\in\RR^d$ are irrational, and suppose that for any pair $W,V$ of words of length $d$
    \[T(W,\bar{\al};1)<T(V,\bar\al;1) \Leftrightarrow T(W,\bar{\beta};1)<T(V,\bar\beta;1),\]
and further suppose that if $W$ and $V$ have different Hamming weight, then $T(W,\bar{\al};1) \not= T(V,\bar\al;1)$.
If
    \[\forall n \geq 1, \qquad T(0^d,\bar{\al};n)= T(0^d,\bar\beta;n),\]
then $\{\fracpart{\al_i} \colon 1\leq i \leq d\}=\{\fracpart{\beta_i}\colon 1\leq i \leq d\}$.
\end{lemma}

\begin{proof}
Set $\Delta(n)=T(0^d,\bar\al;n+1)-T(0^d,\bar\al;n)$, and note that
    \[\{\Delta(n):n\in\NN\}=\{T(W,\bar{\al};1)\colon  \len(W)=d\}.\]
In fact, by the irrationality of $\bar\al$, the density of $n$ such that $\Delta(n)=T(w_1\cdots w_d,\bar{\al};1)$ is
    $$V_W(\bar \al) = \prod_{\substack{i=1\\ w_i=0}}^{d} \fracpart{\al_i} \prod_{\substack{i=1\\w_i=1}}^{d} (1-\fracpart{\al_i}).$$
While for any particular $W$ it is possible that $V_W(\bar\al)\not=V_W(\bar\beta)$, the condition on the ordering of $T(W,\bar\al;1),T(W,\bar\beta;1)$ guarantees the set equalities for $1\le i \le d$:
    \[ \bigg\{ V_W(\bar\al) \colon h(W)=i\bigg\} =  \bigg\{ V_W(\bar\beta) \colon h(W)=i\bigg\}.\]
Thus the polynomial
    \begin{align*}
    P(z)=\prod_{i=1}^{d} \bigg(\fracpart{\al_i}  +(1-\fracpart{\al_i}) z \bigg)
            &= \sum_{I\subseteq\{1,\dots,d\}}\,\prod_{i\in I} \fracpart{\al_i}\,\prod_{i\in\{1,\dots,d\}\setminus I} (1-\fracpart{\al_i}) z \\
            &= \sum_{\substack{W \\ \len(W)=d}}\,\left( V_{W}(\bar\al)\,\prod_{\substack{i=1\\w_i=0}}^{d} 1 \prod_{\substack{i=1\\w_i=1}}^{d} z\right) \\
            &= \sum_{\substack{W \\ \len(W)=d}}\, V_{W}(\bar\al)\,z^{h(W)} \\
            &= \sum_{i=0}^d \left( \sum_{\substack{W \\ \len(W)=d,\, h(W)=i}} V_W(\bar\al) \right) z^{d},
    \end{align*}
is determined by the sequence. Therefore, the set of its roots $- \frac{\fracpart{\al_i}}{1-\fracpart{\al_i}}$ is also determined by the sequence. Since $x\mapsto -\frac{1-x}{x}$ is a $1$-$1$ map, this implies that the set $\{\fracpart{\al_1},\dots,\fracpart{\al_d}\}$ is determined from the sequence, concluding the proof.
\end{proof}

\begin{proof}[Proof of Theorem~\ref{Snested2}]
Combine Lemmas~\ref{lem:Bigalpha} and~\ref{lem:BigDeal}.
\end{proof}

\section{Open questions concerning generalized polynomials}

The meta-issue is to find an efficient algorithm that will determine whether a generalized polynomial with algebraic coefficients is identically zero on the positive integers. Humble first steps in this direction would be to completely answer the problems implied in Concrete Mathematics~\cite{CM}:
\begin{problem}
Find a necessary and sufficient condition on the real numbers $\al_i,\beta_j\in[0,1)$ such that for all positive integers $n$
    \[\sum_{i=1}^d \floor{n\alpha_i} = \sum_{j=1}^\ell \floor{n\beta_j}.\]
\end{problem}
We suspect that this equality happens only if $d=\ell$ and for some $a,b,c,d$, $\al_a+\al_b=\beta_c+\beta_d =1$, and that this (and trivial solutions) are the only way that equality can occur.
\begin{problem}
Find a necessary and sufficient condition on the real numbers $\al_i,\beta_j\in\RR$ such that for all positive integers $n$
    \[\floor{ \cdots \floor{\floor{n\al_1}\,\al_2}\,\cdots \al_d} = \floor{ \cdots \floor{\floor{n\beta_1}\,\beta_2}\,\cdots \beta_\ell} .\]
\end{problem}
There are very many solutions in rationals, and we do not have a guess as to their structure.

Both problems are obvious if all $\al,\beta$ are taken to be integers, and both are answered here if $d=\ell$ and the $\al,\beta$ are taken to be sufficiently irrational. The most difficult case to understand, for both questions, seems to be when the $\al_,\beta$ are all rational, but not all integral.

\section*{Acknowledgements}
We thank Inger H{\aa}land Knutson for helpful comments and nice examples.

\end{document}